\documentclass[12pt]{amsart}

\usepackage{mathrsfs}
\usepackage{amssymb,amsmath,amsxtra,amscd, feynmf,latexsym}
\usepackage{pst-all}
\usepackage{setspace,ifpdf}
\usepackage{diagrams}

\usepackage{youngtab, young}

\numberwithin{equation}{section}

\usepackage{pst-tree, pst-node}
\usepackage[mathscr]{eucal} 
\ifpdf
\usepackage{pdfsync}
\fi

\newcommand{\nc}{\newcommand}
\newcommand{\mc}{\mathcal}
\nc{\on}{\operatorname}
\nc{\h}{\mathfrak{h}}
\nc{\g}{\mathfrak{g}}
\nc{\n}{\mathfrak{n}}
\nc{\ch}{\on{CH}}
\nc{\wt}{\widetilde}
\renewcommand{\P}{\mc{P}}
\nc{\F}{\mc{F}}
\nc{\C}{\mc{C}}
\nc{\M}{\on{M}}
\nc{\T}{\mc{T}}
\renewcommand{\H}{\on{H}}
\nc{\G}{\mc{G}}

\renewcommand{\k}{k}

\theoremstyle{definition}

\newtheorem{theorem}{Theorem}

\newtheorem{definition}{Definition}
\newtheorem{remark}{Remark}

\theoremstyle{definition}
\newtheorem{example}{Example}

\begin{document}

\title{Pre-Lie algebras and Incidence Categories of Colored Rooted Trees}
\author{Matt Szczesny} \thanks{The author is supported by an NSA grant}
\address{Department of Mathematics  and Statistics, 
         Boston University, Boston MA, USA}
\email{szczesny@math.bu.edu}

\begin{abstract}
The incidence category $\C_{\F}$ of a family $\F$ of colored posets closed under disjoint unions and the operation of taking convex sub-posets was introduced by the author in \cite{Sz}, where the Ringel-Hall algebra $\H_{\F}$ of $\C_{\F}$ was also defined. We show that if the Hasse diagrams underlying $\F$ are rooted trees, then the subspace $\n_{\F}$ of primitive elements of $\H_{\F}$ carries a pre-Lie structure, defined over $\mathbb{Z}$, and with positive structure constants. We give several examples of $\n_{\F}$, including the nilpotent subalgebras of $\mathfrak{sl}_n$, $L \mathfrak{gl}_n$, and several others. 
\end{abstract}

\maketitle

\section{Introduction}

 A \emph{left pre-Lie algebra} is a  $\k$--vector space $A$ endowed with a binary bilinear operation $\rhd$ satisfying the identity
\begin{equation} \label{pL}
(a \rhd b) \rhd c - a \rhd ( b \rhd c) = (b \rhd a) \rhd c - b \rhd (a \rhd c)
\end{equation}
It follows easily from \ref{pL} that anti-symmetrizing $\rhd$ yields a Lie bracket $$[a,b] = a \rhd b - b \rhd a$$ on $A$. However, not every Lie algebra arises from a pre-Lie algebra. 
Pre-Lie algebras first appeared in the works of E.B. Vinberg \cite{V} and M. Gerstenhaber \cite{G}, and have since found applications in several areas. One prominent example is 
perturbative quantum field theory \cite{K}, where insertion of Feynman graphs into each other equips them with a pre-Lie structure which controls the combinatorics of the renormalization procedure.  

In this paper, we show that pre-Lie algebras arise naturally from \emph{incidence categories} introduced by that author in \cite{Sz}. An incidence category is built from a collection $\F$ of colored posets, which is closed under the operations of disjoint union and convex subposet - we will denote it by $\C_{\F}$. The objects of $\C_{\F}$ are the posets in $\F$, and for $P_1, P_2 \in \F$
\[
\on{Hom}(P_1, P_2) := \{ (I_1, I_2, f) \vert I_j \textrm{ is an order ideal in } P_j, f: P_1 \backslash I_1 \rightarrow I_2 \textrm{ an isomorphism } \}
\]
Here, the poset $I_1$ should be viewed as the kernel of the morphism, and $I_2$ as the image.  All morphisms in $\C_{\F}$ have kernels and cokernels, and so the notion of exact sequence makes sense. In \cite{Sz}, the Ringel-Hall algebra $\H_{\C_{\F}}$ of $\C_{\F}$ was defined. $\H_{\C_{\F}}$ is the 
$\mathbb{Q}$--vector space of finitely supported functions on isomorphism classes of $\C_{\F}$:
\[
\H_{\C_{\F}} := \{ f: \on{Iso}(\C_{\F}) \rightarrow \mathbb{Q} | |supp(f)| < \infty  \}
\]
with product given by convolution:
\begin{equation} \label{prod}
f \star g (M) = \sum_{A \subset M} f(A) g(M/A). 
\end{equation}
$\H_{\C_{\F}}$ possesses a co-commutative co-product given by
\begin{equation} \label{coppp}
\Delta(f)(M,N)=f(M \oplus N) 
\end{equation}
(where $M \oplus N$ denotes the disjoint union of $M$ and $N$)
as well as an antipode, making it a Hopf algebra. $\H_{\C_{\F}}$ is graded, connected, and co-commutative, and so by the Milnor-Moore theorem isomorphic to $U(\n_{\F})$, where $\n_{\F}$ is the Lie algebra of its primitive elements. It follows from \ref{coppp} that 
\[
\n_{\F} = \on{span} \{ \delta_{P} \vert P \in \F, P \textrm{ connected } \}
\]

We show that if $\F$ consists of posets whose Hasse diagrams are rooted trees, then $\n_{\F}$ carries a pre-Lie structure $\rhd$, with 
\[
\delta_{P_1} \rhd \delta_{P_2} := \delta_{P_1} \star \delta_{P_2} - \delta_{P_1 \oplus P_2}.
\]
A more concrete description of $\rhd$ is the following: $ \delta_{P_1} \rhd \delta_{P_2} $ is a sum of delta-functions supported on connected posets $P \in \F$ whose Hasse diagram is obtained by grafting the root of $P_1$ onto a vertex of $P_2$. It follows from the definition of $\rhd$ that the structure constants are non-negative integers. 

The paper is organized as follows. Section \ref{pL_section} recalls the definition of pre-Lie algebra and introduces the universal example, namely the pre-Lie algebra of colored rooted trees. In section \ref{icat_section} we recall the construction of the incidence category $\C_{\F}$ as well as its main properties. The Ringel-Hall algebra of $\C_{\F}$ is introduced in section \ref{Hall_section}. In section \ref{main_section} we define the pre-Lie structure $\rhd$ on $\n_{\F}$ and verify that it satisfies the identity \ref{pL}. Finally, section \ref{example_section} is devoted to examples - among these are pre-Lie structures on nilpotent Lie subalgebras of $\mathfrak{sl}_n$ and $L{\mathfrak{gl}}_n$.

\bigskip

\noindent {\bf Acknowledgements: } The author is very grateful to Pavel Etingof for valuable discussions and suggestions. 

\section{Pre-Lie algebras} \label{pL_section}

In this section, we recall the definition and some examples of (left) pre-Lie algebras. Let $\k$ be a field. 

\begin{definition}
A \emph{left pre-Lie algebra} is a  $\k$--vector space $A$ endowed with a binary bilinear operation $\rhd$ satisfying the left pre-Lie identity
\begin{equation} \label{preLieIdentity}
(a \rhd b) \rhd c - a \rhd ( b \rhd c) = (b \rhd a) \rhd c - b \rhd (a \rhd c)
\end{equation}
for $a,b,c \in A$.
\end{definition}
One checks easily that antisymmetrizing the operation $\rhd$ 
\[
[a,b] = a \rhd b - b \rhd a
\]
gives $A$ the structure of a Lie algebra. 

\begin{example}
Any associative $k$--algebra $A$ is a pre-Lie algebra with the pre-Lie structure given by
\[
a \rhd b := ab,
\]
where the right hand side refers to the associative multiplication in $A$. 
\end{example}

\begin{example} \label{RT}
One of the most important examples of pre-Lie algebras is given by colored rooted trees. Recall that a \emph{tree} is a graph with no cycles. We denote by $E(t), V(t)$ the edge and vertex sets of $t$ respectively. Let $S$ be a finite set.  By a \emph{rooted tree colored by S} we mean a tree with a distinguished vertex $r(t) \in V(t)$ called the \emph{root}, and an assignment of an element of $S$ to each $v \in V(t)$. We adopt the convention that rooted trees are always drawn with the root on top. For example, if $S =\{a,b \}$, then the following are rooted trees colored by $S$:

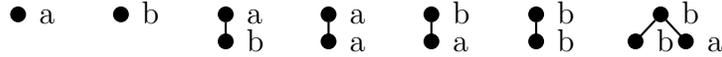
\psset{levelsep=2ex, treesep=0.6cm, treenodesize=1pt,nodesepB=-2pt}

\begin{center} \pstree{\Tc*{3pt}~[tnpos=r]{a}}{} \hspace{0.5cm} \pstree{\Tc*{3pt}~[tnpos=r]{b}}{} \hspace{0.5cm}  \pstree{\Tc*{3pt}~[tnpos=r]{a}}{\Tc*{3pt}~[tnpos=r]{b}} \hspace{0.5cm} \pstree{\Tc*{3pt}~[tnpos=r]{a}}{\Tc*{3pt}~[tnpos=r]{a}} \hspace{0.5cm} \pstree{\Tc*{3pt}~[tnpos=r]{b}}{\Tc*{3pt}~[tnpos=r]{a}} \hspace{0.5cm} 
\pstree{\Tc*{3pt}~[tnpos=r]{b}}{\Tc*{3pt}~[tnpos=r]{b}} \hspace{0.5cm} \pstree{\Tc*{3pt}~[tnpos=r]{b}}{\Tc*{3pt}~[tnpos=r]{b} \Tc*{3pt}~[tnpos=r]{a}}
\end{center}

Let $\mathbb{T}_S$ denote the set of rooted trees whose vertices are colored by $S$. Given $t \in \mathbb{T}_S$, and $e \in E(t)$, removing $e$ disconnects $t$ into two colored rooted trees: $R_e(t)$ containing $r(t)$ and $P_e(t)$, whose root is the end of $e$. Let $\T_S$ be the $\k$--vector space spanned by $\mathbb{T}_S$. We have 
\[
\T_S = \oplus^{\infty}_{n=0} \T_S [n]
\]
where $\T_S [n]$ is the subspace of $\T_S$ spanned by trees with $n$ vertices. For colored rooted trees $t_1, t_2 \in \mathbb{T}_S$, let 
\[
t_1 \rhd t_2 := \sum_{s \in \mathbb{T}_S} n(t_1, t_2, s) s
\]
where 
\[
n(t_1, t_2,s) = \# \{ e \in E(s) | P_e (s) = t_1, R_e(s) = t_2 \}.
\]
For example, we have
\[
 \pstree{\Tc*{3pt}~[tnpos=r]{b}}{} \rhd   \pstree{\Tc*{3pt}~[tnpos=r]{a}}{\Tc*{3pt}~[tnpos=r]{b}} =  2  \pstree{\Tc*{3pt}~[tnpos=r]{a}}{\Tc*{3pt}~[tnpos=r]{b} \Tc*{3pt}~[tnpos=r]{b}} + \pstree{\Tc*{3pt}~[tnpos=r]{a}}{\pstree{\Tc*{3pt}~[tnpos=r]{b}}{ \Tc*{3pt}~[tnpos=r]{b}}}
\]
It is well-known (see for instance \cite{CL}) that $\rhd$ defines a pre-Lie structure on $\T_S$. The following theorem is proven in \cite{CL}

\begin{theorem}
$\T_S$ is the free pre-Lie algebra on $|S|$ generators. 
\end{theorem}

\begin{remark}
In what follows, unless stated otherwise, $k=\mathbb{Q}$. 
\end{remark}

\end{example}

\section{Incidence categories} \label{icat_section}

\subsection{Recollections on posets} 

We begin by recalling some basic notions and terminology pertaining to posets ( partially ordered sets) following \cite{Sch, St}. 

\begin{enumerate}
\item An \emph{interval} is a poset having unique minimal and maximal elements. For $x, y$ in a poset $P$, we denote by $[x,y]$ the interval
\[
[x,y] := \{ z \in P : x \leq z \leq y \}
\]  
If $P$ is an interval, we will often denote by $0_P$ and $1_P$ the minimal and maximal elements. 
\item An \emph{order ideal} in a poset $P$ is a subset $L \subset P$ such that whenever $y \in L$ and $x \leq y$ in $P$, then $x \in L$.
\item  A sub-poset $Q$ of $P$ is \emph{convex} if, whenever $x \leq y$ in $Q$ and $z \in P$ satisfies $x \leq z \leq y$, then $z \in Q$. Equivalently, $Q$ is convex if $Q = L \backslash I$ for order ideals $I \subset L$ in $P$. 
\item Given two posets $P_1, P_2$, their disjoint union is naturally a poset, which we denote by $P_1 + P_2$. In $P_1 + P_2$, $x \leq y$ if both lie in either $P_1$ or $P_2$, and $x \leq y$ there.
\item A poset which is not the union of two non-empty posets is said to be \emph{connected}. 
\item The cartesian product $P_1 \times P_2$ is a poset where $(x,y) \leq (x',y')$ iff $x \leq x'$ and $y \leq y'$. 
\item A \emph{distributive lattice} is a poset $P$ equipped with two operations $\wedge$, $\vee$ that satisfy the following properties:
\begin{enumerate}
\item $\wedge, \vee$ are commutative and associative
\item $\wedge, \vee$ are idempotent - i.e. $x \wedge x = x$, $x \vee x = x$
\item $x \wedge (x \vee y) = x = x \vee ( x \wedge y)$
\item $x \wedge y = x \iff x \vee y = y \iff x \leq y$
\item $ x \vee (y \wedge z) = (x \vee y) \wedge (x \vee z)$
\item $x \wedge (y \vee z) = (x \wedge y) \vee (x \wedge z)$
\end{enumerate}

\item For a poset $P$, denote by $J_P$ the poset of order ideals of $P$,  ordered by inclusion. $J_P$ forms a distributive lattice with $I_1 \vee I_2 := I_1 \cup I_2$ and $I_1 \wedge I_2 := I_1 \cap I_2$ for $I_1, I_2 \in J_P$. If $P_1, P_2$ are posets, we have $J_{P_1 + P_2} = J_{P_1} \times J_{P_2}$, and if $I,L \in J_P$, and $I \subset L$, then $[I,L]$ is naturally isomorphic to the lattice of order ideals $J_{L \backslash I}$. 

\end{enumerate}

\bigskip

\begin{remark}
Suppose that the Hasse diagram of a poset $P$ is a rooted tree - that is, $P$ has a unique maximal element $r(P)$, and the Hasse diagram contains no cycles. It is  then easy to see that order ideals $I \subset P$ correspond to \emph{admissible cuts} of $P$, where the latter is a collection of edges $C \subset E(P)$, having the property that at most one edge of $C$ is encountered along any path from root to leaf. For instance, the dotted edges of the poset $T$ below yield an admissible cut:
\psset{levelsep=3ex, treesep=0.6cm, treenodesize=1pt,nodesepB=-2pt}
\[
\pstree{\Tc*{3pt}~[tnpos=r]{b}}{ \psset{linestyle=dashed} \Tc*{3pt}~[tnpos=r]{b} \psset{linestyle=solid} \pstree{\Tc*{3pt}~[tnpos=r]{a}}{ \psset{linestyle=dashed}  \Tc*{3pt}~[tnpos=r]{b}} }
\]
Each admissible cut $C \subset E(P)$ divides the tree into a rooted connected tree $R_C(P)$ containing $r(P)$, and a rooted forest (a disjoint union of rooted trees) $P_C (P)$. The notation is clearly an extension of that used in example \ref{RT}. In the last example, we have
\[
R_C (T) = \pstree{\Tc*{3pt}~[tnpos=r]{b}}{\Tc*{3pt}~[tnpos=r]{a}} \hspace{2cm} P_C(T) =  \pstree{\Tc*{3pt}~[tnpos=r]{b}}{} \hspace{0.3cm} \pstree{\Tc*{3pt}~[tnpos=r]{b}}{} 
\]
\end{remark}

\subsection{From posets to categories}

Let $\F$ be a family of colored posets which is closed under the formation of disjoint unions and the operation of taking convex subposets, and let 
\mbox{$\P(\F) = \{ J_P : P \in \F \}$} be the corresponding family of distributive lattices of order ideals. For each pair 
$P_1, P_2 \in \F$, let $\M(P_1, P_2)$ denote the set of colored poset isomorphisms $P_1 \rightarrow P_2$. It follows that $\M(P,P)$ forms a group, which we denote $\on{Aut}_{\M}(P)$. 

\subsubsection{The category $\C_{\F}$}
\vspace{0.5cm}

We proceed to define a category $\C_{\F}$, called the \emph{incidence category of $\F$} as follows. Let
$$\on{Ob}(\C_{\F} ) := \F = \{  P \in \F \}$$ and 
$$ \on{Hom}({P_1}, {P_2}) := \{ ( I_1, I_2, f) : I_i \in J_{P_i},  f \in \M(P_1 \backslash I_1, I_2) \} \; \; i=1,2$$
We need to define the composition of morphisms
\[
\on{Hom}({P_1}, {P_2}) \times \on{Hom}({P_2}, {P_3}) \rightarrow \on{Hom}({P_1}, {P_3})
\]
Suppose that $(I_1, I_2, f) \in \on{Hom}({P_1}, {P_2})$ and $(I'_2, I'_3, g) \in \on{Hom}({P_2}, {P_3})$. Their composition is the morphism $(K_1, K_3, h)$ defined as follows. 
\begin{itemize}
\item We have $I_2 \wedge I'_2 \subset I_2$, and since $f: P_1 \backslash I_1 \rightarrow I_2$ is an isomorphism, $f^{-1}(I_2 \wedge I'_2)$ is an order ideal of $P_1 \backslash I_1$. Since in $J_{P_1}$, $[I_1, P] \simeq J_{P_1 \backslash I_1}$, we have that $f^{-1}(I_2 \wedge I'_2)$ corresponds to an order ideal $K_1 \in J_{P_1}$ such that $I_1 \subset K_1$. 
\item We have $I'_2 \subset I_2 \vee I'_2$, and since $ [I'_2, P_2] \simeq J_{P_2 \backslash I'_2}$, $I_2 \vee I'_2$ corresponds to an order ideal $L_2 \in J_{P_2 \backslash I'_2}$. Since $g: P_2 \backslash I'_2 \rightarrow I'_3$ is an isomorphism, $g(L_2) \subset J_{I'_3}$, and since $J_{I'_3} \subset J_{P_3}$, $g(L_2)$ corresponds to an order ideal $K_3 \in J_{P_3}$ contained in $I'_3$. 
\item  The isomorphism $f: P_1 \backslash I_1 \rightarrow I_2$ restricts to an isomorphism $\bar{f}: P_1 \backslash K_1 \rightarrow I_2 \backslash I_2 \wedge I'_2 = I_2 \backslash I'_2 $, and the isomorphism $g: P_2 \backslash I'_2$ restricts to an isomorphism $\bar{g}: I_2 \vee I'_2 \backslash I'_2 = I_2 \backslash I'_2 \rightarrow  K_3$. Thus, $g \circ f: P_1 \backslash K_1 \rightarrow K_3$ is an isomorphism and $g \circ f \in \M(P_1 \backslash K_1, K_3)$ by the property $(4)$ above. 
 \end{itemize}

As shown in \cite{Sz}, the composition of morphisms is associative. 

\bigskip

\begin{remark}

\noindent 
\begin{itemize}
\item We refer to ${I_2}$ as the \emph{image} of the morphism $(I_1, I_2, f): {P_1} \rightarrow {P_2}$. 
\item We denote by $\on{Iso}(\C_{\F})$ the collection of isomorphism classes of objects in $\C_{\F}$, and by $[P]$ the isomorphism class of $P \in \C_{\F}$. 
\end{itemize}

\end{remark}

\subsection{Properties of the categories $\C_{\F}$}

We now enumerate some of the properties of the categories $\C_{\F}$. 
\bigskip
\begin{enumerate}
\item The empty poset $\emptyset$ is an initial, terminal, and therefore null object. We will sometimes denote it by ${\emptyset}$.   
\bigskip
\item We can equip $\C_{\F}$ with a symmetric monoidal structure by defining
\[
{P_1} \oplus {P_2}  := {P_1 + P_2}.
\]
\item The indecomposable objects of $\C_{\F}$ are the $P$ with $P$ a connected poset in $\F$.
\bigskip
\item The simple objects of $\C_{\F}$ are the ${P}$ where $P$ is a one-element poset. 
\bigskip
\item  \label{kernel} Every morphism 
\begin{equation} \label{morphism}
(I_1, I_2,f) : {P_1} \rightarrow {P_2}
\end{equation}
has a kernel
\[
(\emptyset,I_1, id):  {I_1} \rightarrow {P_1} 
\]
\bigskip
\item \label{cokernel} Similarly, every morphism \ref{morphism} possesses a cokernel
\[
(I_2, P_2 \backslash I_2, id) : {P_2} \rightarrow P_2 \backslash I_2
\]
\medskip
\noindent We will use the notation ${P_2}/{P_1}$ for $coker((I_1,I_2,f))$. 

\bigskip
\noindent {\bf Note:} Properties \ref{kernel} and \ref{cokernel} imply that the notion of exact sequence makes sense in $\C_{\F}$.
\bigskip
\item All monomorphisms are of the form
\[
(\emptyset,I, f) : Q \rightarrow {P}
\]
where $I \in J_P$, and $f: Q \rightarrow I \in \M(Q,I)$. Monomorphisms $Q \rightarrow P$ with a fixed image ${I}$ form a torsor over $\on{Aut}_{\M}(I)$. 
%where $f$ is an automorphism of $P_{C_1}(F_1)$. Once the image subforest $P_{C_1}(F_1)$ is fixed, all monomorphisms with that image form a torsor over $\on{Aut}(P_{C_1})$, and there are therefore $|\on{Aut}P_{C_1}(F_1)|$ of them. 
All epimorphisms are of the form
\[
(I,\emptyset, g) : P \rightarrow {Q}
\]
where $I \in J_P$ and $g: P \backslash I \rightarrow Q \in \M(P \backslash I, Q)$. Epimorphisms with fixed kernel $I$ form a torsor over $\on{Aut}_{\M}(P \backslash I)$ 
\bigskip
\item
Sequences of the form \label{propses}
\begin{equation} \label{ses}
{\emptyset} \overset{(\emptyset, \emptyset, id)}{\rightarrow} {I} \overset{(\emptyset,I, id)}{\longrightarrow} {P} \overset {(I,\emptyset,id)}{\longrightarrow} {P \backslash I} \overset{(P \backslash I, \emptyset, id)}{\rightarrow} \emptyset
\end{equation} 
with $I \in J_P$ are short exact, and all other short exact sequences with $P$ in the middle arise by composing with isomorphisms $I \rightarrow {I'}$ and ${P \backslash I} \rightarrow Q$ on the left and right. 
\bigskip
\item

 \label{quotients}  Given an object ${P}$ and a subobject ${I}, I \in J_P$,  the isomorphism $J_{P \backslash I} \simeq [I,P]$ translates into the statement that there is a bijection between subobjects of $P/I$ and order ideals $J \in J_P$ such that $I \subset J \subset P$. The bijection is compatible with quotients, in the sense that $(P/I)/(J/I) \simeq J/I$. 
 
 %given a forest $F$ and an admissible cut $C$, there is a bijection between subobjects $F'$ of $F$ containing $P_C(F)$, i.e. chains $P_{C}(F) \subset F' \subset F$, and subobjects of $R_C(F)$. 

\bigskip
\item Since the posets in $\F$ are finite, $\on{Hom}({P_1},{P_2})$ is a finite set. 
\bigskip
\item We may define Yoneda $\on{Ext}^{n}({P_1}, {P_2})$ as the equivalence class of $n$--step exact sequences with ${P_1}, {P_2}$ on the right and left respectively.  $\on{Ext}^{n} ({P_1}, {P_2})$ is a finite set. Concatenation of exact sequences makes $$\mathbb{E}xt^* := \cup_{A,B \in I(\C_{\F}), n} \on{Ext}^n (A,B)$$ into a monoid.  
\bigskip
\item We may define the Grothendieck group of $\C_{\F}$, $K_0(\C_{\F})$, as 
\[
K(\C_{\F}) = \bigoplus_{A \in \C_{\F}} \mathbb{Z}[A] / \sim
\]
where $\sim$ is generated by  $A+B-C$ for short exact sequences
\[
{\emptyset} \rightarrow A \rightarrow C \rightarrow B \rightarrow {\emptyset}
\]
We denote by $k(A)$ the class of an object in $K_0(\C_{\F})$. 
\end{enumerate}

\section{Ringel-Hall algebras} \label{Hall_section}

 For an introduction to Ringel-Hall algebras in the context of abelian categories, see \cite{S}. 
We define the Ringel-Hall algebra of $\C_{\F}$, denoted $\H_{\C_{\F}}$, to be the 
$\mathbb{Q}$--vector space of finitely supported functions on isomorphism classes of $\C_{\F}$. I.e.
\[
\H_{\C_{\F}} := \{ f: \on{Iso}(\C_{\F}) \rightarrow \mathbb{Q} | |supp(f)| < \infty  \}
\]
As a $\mathbb{Q}$--vector space it is spanned by the delta functions $\delta_A, A \in \on{Iso}(\C_{\F})$. The algebra structure on $\H_{\C_{\F}}$ is given by the convolution product:
\begin{equation} \label{prod}
f \star g (M) = \sum_{A \subset M} f(A) g(M/A) 
\end{equation}
for $M \in \on{Iso}(\C_{\F})$. In what follows, it will be conceptually useful to choose a representative in each isomorphism class. For $M, N, Q \in \on{Iso}(\C_{\F})$, let $F^Q_{M,N}$ be the number of exact sequences
\[
\emptyset \rightarrow M \overset{i}{\rightarrow}  Q \overset{\pi}{\rightarrow} N \rightarrow \emptyset
\]
where $(i,\pi)$ and $(i', \pi')$ are considered equivalent iff $i=i'$ and $\pi=\pi'$ (this makes sense, since we have fixed a representative in each isomorphism class). It follows from the definition \ref{prod} that
\[
\delta_M \star \delta_N = \sum_{Q \in \on{Iso}(\C_{\F})} \frac{F^{Q}_{M,N}}{|\on{Aut}(M)| |\on{Aut}(N) | }  \delta_Q,
\]
from which it is apparent that $\H_{\C_{\F}}$ encodes the structure of extensions in $\C_{\F}$. 

$\H_{\C_{\F}}$ possesses a co-commutative co-product given by
\begin{equation} \label{cop}
\Delta(f)(M,N)=f(M \oplus N) 
\end{equation}
as well as a natural $K^+_0 (\C_{\F})$--grading in which  $\delta_A$ has degree $k(A) \in K^+_0 (\C_{\F})$. If $\F$ is colored by the set $S$, it is easy to see that $K^+_0 (\C_{\F}) \simeq \mathbb{N}^{|S|}$. 

The subobjects of  $P \in \C_{\F}$ are exactly $I \in J_P$, and the product \ref{prod} becomes
\[
f \star g ([P]) = \sum_{I \in J_P} f([I]) g([P \backslash I]).
\]
 It is shown in \cite{S} that the product is associative, the co-product co-associative and co-commutative, and that the two are compatible, making $\H_{\C_{\F}}$ into a co-commutative bialgebra.  Recall that a bialgebra $A$ over a field $k$ is \emph{connected} if it possesses a $\mathbb{Z}_{\geq 0}$--grading such that $A_0 = k $. In addition to the $K^+_0(\C_{\F})$--grading, $\H_{\C_{\F}}$ possesses a grading by the order of the poset - i.e. we may assign $\deg(\delta_{P}) = |P|$. This gives it the structure of graded connected bialgebra, and hence Hopf algebra.  The Milnor-Moore theorem implies that $\H_{\C_{\F}}$ is the enveloping algebra of the Lie algebra of its primitive elements, which we denote by $\n_{\F}$ - i.e. $\H_{\C_{F}} \simeq U(\n_{\F})$. It follows from \ref{cop} that $f \in \n_{\F}$ is primitive if it is supported on the isomorphism classes of connected posets. Thus, we have that
 \[
 \n_{\F} = \on{span} \{ \delta_P \vert P \in \F, P \textrm{ connected } \}
 \]
 We will use the notation $\F^{conn} \subset \F$ to denote the sub-collection of $\F$ consisting of connected posets. 
 We have thus established the following:
 \begin{theorem}
 The Ringel-Hall algebra of the category $\C_{\F}$ is a co-commutative graded connected Hopf algebra, isomorphic to $U(\n_{\F})$, where $\n_{\F}$ denotes the graded Lie algebra of its primitive elements.  $\n_{\F} = \on{span} \{ \delta_{P} \vert P \in \F^{conn} \}$.
 \end{theorem}
 
 \begin{remark}
 $\H_{\C_{\F}}$ is a special case of an \emph{incidence Hopf algebra} introduced by Schmitt in \cite{Sch,Sch2}. 
 \end{remark}

\section{A Pre-Lie structure on $\n_{\F}$} \label{main_section}

We assume now that the collection $\F$ consists of colored posets whose underlying Hasse diagrams are rooted trees. 
%Observe that order ideals in a rooted tree correspond to \emph{admissible cuts}. 
Recall that $\F$ was assumed to be:
\begin{itemize}
\item closed under the operation of taking convex sub-posets
\item closed under disjoint unions
\end{itemize}
It is immediate that to produce an $\F$ satisfying these two requirements, one may start with an arbitrary collection $\F'$ of colored posets, and close it with respect to each operation - i.e. adjoin to $\F'$ all convex sub-posets and all disjoint unions of these. If $\F$ arises in this way as the closure of $\F'$, we will write $\F=\overline{\F'}$. 

\begin{example} \label{ex1ton}
Suppose that $\F'$ consists of a single poset, whose Hasse diagram is an $n$--vertex ladder colored by the set $S=\{1, \cdots, n \}$. 
\psset{levelsep=4ex, treesep=0.6cm, treenodesize=1pt,nodesepB=-2pt}
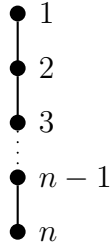
\[
\pstree{\Tc*{3pt}~[tnpos=r]{1}}{ \pstree{\Tc*{3pt}~[tnpos=r]{2}}{ \pstree{\Tc*{3pt}~[tnpos=r]{3}}{ \psset{linestyle=dotted}   \pstree{\Tc*{3pt}~[tnpos=r]{n-1}}{ \psset{linestyle=solid}  \pstree{\Tc*{3pt}~[tnpos=r]{n}}{}    }      }} }
 \]
Let us adopt the notation $L(a_1, a_2, ..., a_k)$ for a $k$--vertex ladder Hasse diagram labeled by $a_1, a_2, ..., a_k$ root-to-leaf ($\F'$ thus consisting of $L(1,2,\cdots, n)$). To close $\F'$ with respect to convex subsets, we must adjoin to it $L(r,r+1,r+2, \cdots, r+m)$, where $1 \leq r \leq r+m \leq n$. 
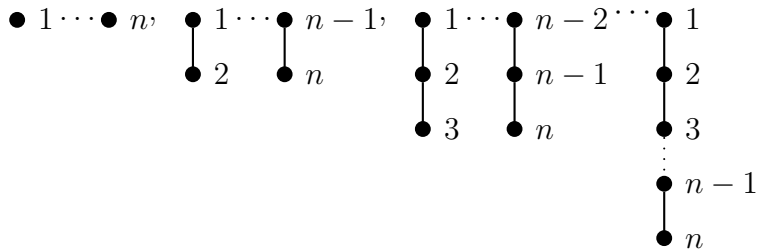
\[
\pstree{\Tc*{3pt}~[tnpos=r]{1}}{} \dots \pstree{\Tc*{3pt}~[tnpos=r]{n}}{}, \hspace{0.3cm} \pstree{\Tc*{3pt}~[tnpos=r]{1}}{ \Tc*{3pt}~[tnpos=r]{2} } \dots \pstree{\Tc*{3pt}~[tnpos=r]{n-1}}{ \Tc*{3pt}~[tnpos=r]{n} } , \hspace{0.3cm}  \pstree{\Tc*{3pt}~[tnpos=r]{1}}{\pstree{\Tc*{3pt}~[tnpos=r]{2}}{ \Tc*{3pt}~[tnpos=r]{3}}} \dots  \pstree{\Tc*{3pt}~[tnpos=r]{n-2}}{\pstree{\Tc*{3pt}~[tnpos=r]{n-1}}{ \Tc*{3pt}~[tnpos=r]{n}}} \cdots \pstree{\Tc*{3pt}~[tnpos=r]{1}}{ \pstree{\Tc*{3pt}~[tnpos=r]{2}}{ \pstree{\Tc*{3pt}~[tnpos=r]{3}}{ \psset{linestyle=dotted}   \pstree{\Tc*{3pt}~[tnpos=r]{n-1}}{ \psset{linestyle=solid}  \pstree{\Tc*{3pt}~[tnpos=r]{n}}{}    }      }} }
\]
Finally, closing with respect to disjoint unions, we can identify elements of $\F = \overline{\F'}$ with Young diagrams having at most $n$ rows, each of whose columns is labeled by $k, k+1, \cdots, k+m$. For instance
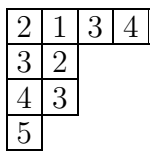
\[
\Yvcentermath1
\young(2134,32,43,5)
\]
is identified with the poset $$L(2,3,4,5) + L(1,2,3) + L(3) + L(4). $$ 
\end{example}

\bigskip

We proceed to equip $\n_{\F}$ with a pre-Lie structure. For $a,b \in \F^{conn}$, we define
\begin{equation} \label{preLieProd}
\delta_a \rhd \delta_b = \delta_a \star \delta_b - \delta_{a \oplus b}
\end{equation}
and extend the product $\rhd$ to all of $\n_{\F}$ by linearity. The subtraction of the term $\delta_{a \oplus b}$ in \ref{preLieProd} has the effect of removing the delta-function supported on the one split extension of $b$ by $a$, and so the right-hand side of \ref{preLieProd} does indeed lie in $\n_{\F}$. It follows easily that we may re-write the definition \ref{preLieProd} as:

\begin{equation} \label{preLieProd2}
\delta_a \rhd \delta_b = \sum_{t \in \F} n(a,b,t) \delta_{t}
\end{equation}
where $n(a,b,t)$ is defined as in example \ref{RT}.

\begin{theorem}
Let $\F$ be a collection of colored posets closed with respect to taking convex sub-posets and disjoint unions. If the Hasse diagrams of posets in $\F$ are rooted trees, then $\rhd$ equips $\n_{\F}$ with the structure of a pre-Lie algebra. 
\end{theorem}

\begin{proof}
A two-sided pre-Lie ideal in a pre-Lie algebra $A$ is a subspace $I \subset A$ such that  if $x \in I$, then $a \rhd x \in I$ and $x \rhd a \in I \; \forall a \in A$. One checks easily that the quotient $A / I$ inherits a pre-Lie structure. Let $\F$ be a collection of colored rooted forests colored by $S$, closed under the operations of disjoint union and convex sub-poset, and $\F^{conn} \subset \F$ the connected ones (i.e. the rooted trees). $\F^{conn}$ is closed under taking convex sub-posets. I claim that $J= \T_{S} \backslash \F^{conn}$ is a two-sided pre-Lie ideal in $\T_{S}$. Let $u \in \T_{S}$ and $s \in J$. We have
\[
\delta_{u} \rhd \delta_s = \sum_{t \in \T_S} n(u,s,t) \delta_{t}
\]
Suppose that $n(u,s,t) \neq 0$ and $t \in \T_S \backslash J = \F^{conn}$. $t$ has an edge $e$ such that $P_e(t) = u$ and $R_{e} (t) = s$, and since both are convex sub-posets of the poset $t \in \F^{conn}$, $u,s \in \F^{conn}$, contradicting the fact that $s \in J$. It follows that $\delta_u \rhd \delta_s \in J$. The same argument shows that $\delta_s \rhd \delta_u \in J$. The quotient $\T_S / J$ is canonically identified with $\n_{\F}$ with the bracket \ref{preLieProd2}. 
\end{proof}

We give a second proof, very close to the one for $\F=\mathbb{T}_S$ given in \cite{CL}.

\begin{proof}
We need to verify the identity \ref{preLieIdentity}. It follows from \ref{preLieProd2} that for $a,b,c \in \F^{conn}$,

\begin{align*}
(\delta_a \rhd \delta_b ) \rhd \delta_c & = (\sum_{t \in \F^{conn}} n(a,b,t) t ) \rhd c \\
                                                           & = \sum_{s,t \in \F^{conn}} n(a,b,t) n(t,c,s) s \\
                                                           & \textrm{ and } \\
\delta_a \rhd ( \delta_b  \rhd \delta_c ) & = \sum_{s,t \in \F^{conn}} a \rhd (\sum_{t \in \F^{conn}} n(b,c,t) t ) \\
                                                             & = \sum_{s,t \in \F^{conn}} n(b,c,t) n(a,t,s) s
\end{align*}
Because $\F$ is closed under taking convex sub-posets, $P_e(t) \in \F^{conn}$ and $R_{e}(t) \in \F^{conn}$,  $\forall t \in \F^{conn}$. The sum
$\sum_{t \in \F^{conn}} n(a,b,t) n(t,c,s)$ may be identified with the number of pairs of edges $\pi = \{ e_1, e_2 \} \subset E(s)$, such that the resulting cut is NOT admissible (i.e. both edges lie along a single path from root to leaf in $s$), and the three connected components when $\pi$ is removed, are, top-to-bottom, $c, b$ and $a$.  Similarly, the sum $ \sum_{t \in \F^{conn}} n(b,c,t) n(a,t,s) $ may identified with the number of pairs $\pi' = \{e_1, e_2 \} \subset E(s)$ such that the corresponding cut of $s$ results in three components $a,b,c$, with $r(s) \in c$, and no element of $a$ greater than an element of $b$.  The coefficient of $\delta_s$ in 
\[
\delta_a \rhd ( \delta_b  \rhd \delta_c ) - (\delta_a \rhd \delta_b ) \rhd \delta_c 
\] 
therefore counts the number of \emph{admissible} two-edge cuts of $s$ such that the connected component containing $r(s)$ is isomorphic to $c$, and the remaining two to $a, b$ respectively. 

Applying the same analysis to the right-hand-side of \ref{preLieIdentity} proves the equality. 

\end{proof}

\begin{remark}
It follows from \ref{preLieProd2} that $\n_{\F}$ is defined over $\mathbb{Z}$, and that the structure constants are non-negative. 
\end{remark}

\section{Examples} \label{example_section}

In this section, we consider different examples of families $\F$, and the resulting pre-Lie algebras $\n_{\F}$. Recall that since $\n_{\F}$ is graded by $\mathbb{N}$, the Lie algebra $\n_{\F}$ is pro-nilpotent (nilpotent if $\n_{\F}$ is finite-dimensional). 

\begin{example}
Let $S$ be a finite set, and $\F=\overline{\T_S}$, the set of rooted forests colored by $S$. We then obtain the pre-Lie algebra structure on $S$--labeled rooted trees described in example \ref{RT}. 
\end{example}

\begin{example}
Suppose $S$ consists of a single element, and let $\F=\overline{\F'}$, where $\F'$ is the collection of all ladders:
\psset{levelsep=4ex, treesep=0.6cm, treenodesize=1pt,nodesepB=-2pt}
\[
\pstree{\Tc*{3pt}~[tnpos=r]{}}{ \pstree{\Tc*{3pt}~[tnpos=r]{}}{ \pstree{\Tc*{3pt}~[tnpos=r]{}}{ \psset{linestyle=dotted}   \pstree{\Tc*{3pt}~[tnpos=r]{}}{ \psset{linestyle=solid}  \pstree{\Tc*{3pt}~[tnpos=r]{}}{}    }      }} }
\]
(since there is only one color, we suppress the labeling). Denote by $L_n$ the n-vertex ladder. We have
\[
\delta_{L_n} \rhd \delta_{L_m} = \delta_{L_{m+n}}. 
\]
so the Lie algebra $\n_{\F}$ is abelian. In the Ringel-Hall algebra $\H_{\C_{\F}}$ we have
\[
\delta_{L_n} \star \delta_{L_m} = \delta_{L_{m+n}} + \delta_{L_m \oplus L_n}
\]
and
\[
\Delta(L_m) = L_m \otimes 1 + 1 \otimes L_m
\]
It is well-known (see eg. \cite{M}) that the Hopf algebra $\H_{\C_{\F}}$ is isomorphic to the Hopf algebra of symmetric functions, with $L_m$ corresponding to the mth power sum. 
\end{example}

\begin{example}
Let $S = \{ 1, 2, \cdots, n \}$, and let $\F=\overline{\F'}$, where $\F'$ consists of singleton vertices colored by $S$. $\F$ is thus the collection of all finite sets colored by $S$, with trivial partial order. Denote by $X(m_1, m_2, \cdots, m_n)$ the set of $m_1+m_2+\cdots +m_n$ elements, with $m_i$ colored $i$, $1 \leq i \leq n$. $\n_{\F}$ is therefore spanned by the $\delta_{X(0, \cdots, \underset{i}{1}, \cdots, 0)}$. The operation $\rhd$ is identically $0$,
so the Lie algebra $\n_{\F}$ is abelian. In $\H_{\C_{\F}}$ we have
\[
\delta_{X(m_1, \cdots, m_n)} \star \delta_{X(m'_1, \cdots, m'_n)} =  \left( \prod^{n}_{i} {{m_i+m'_i} \choose m_i} \right)  \delta_{ (m_1+m'_1, \cdots, m_n + m'_n)}
\]
\end{example}

\begin{example}
Let $S = \{ 1, 2, \cdots, n \}$, and let $\F=\overline{\F'}$, where $\F'$ consists of all $S$--colored ladder trees
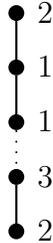
\psset{levelsep=4ex, treesep=0.6cm, treenodesize=1pt,nodesepB=-2pt}
\[
\pstree{\Tc*{3pt}~[tnpos=r]{2}}{ \pstree{\Tc*{3pt}~[tnpos=r]{1}}{ \pstree{\Tc*{3pt}~[tnpos=r]{1}}{ \psset{linestyle=dotted}   \pstree{\Tc*{3pt}~[tnpos=r]{3}}{ \psset{linestyle=solid}  \pstree{\Tc*{3pt}~[tnpos=r]{2}}{}    }      }} }
\]
Denote by $L(a_1, \cdots, a_k)$ the $k$-vertex ladder whose $i$th vertex counting from the \emph{leaf} is colored $a_i$. We have
\begin{equation} \label{pprod}
\delta_{L(a_1, \cdots, a_n)} \rhd \delta_{L(b_1, \cdots, b_m)} = \delta_{L(a_1, \cdots, a_n, b_1, \cdots, b_m)} 
\end{equation}
Let $\mathbb{Q}<X_1, \cdots, X_s>$ denote the free associative algebra on $S$ viewed as a Lie algebra. There is a linear isomorphism 
\begin{align*}
\rho: \n_{\F} & \rightarrow \mathbb{Q}<X_1, \cdots, X_s>\\
\rho(L(a_1, \cdots, a_k)) &= X_{a_1} X_{a_2} \cdots X_{a_k}
\end{align*}
It follows from \ref{pprod} that $\rho$ is a Lie algebra isomorphism. 
\end{example}

\begin{example}
Consider the collection $\F$ from example \ref{ex1ton}, where $\F = \overline{L(1,2, \cdots, n)}'$. Here $\n_{\F} = \on{span} \{ \delta_{L(k, \cdots, k+m)} \}$, $1 \leq k \leq k+m \leq n$.  We have

\[
 \delta_{L(p, \cdots,  p+r)} \rhd \delta_{L(k, \cdots, k+m)} =  \left\{ \begin{array}{ll} 
  \delta_{L(k, \cdots, p+r)} & \mbox{ if $k+m+1=p$ } \\
  0 & \mbox{ otherwise }
\end{array} \right.
\]
so that in the Lie algebra $\n_{\F}$, 
\begin{equation} \label{commNn}
[\delta_{L(p, \cdots,  p+r)}, \delta_{L(k, \cdots, k+m)}] =   \left\{ \begin{array}{ll} 
  \delta_{L(k, \cdots, p+r)} & \mbox{ if $k+m+1=p$ } \\
  0 & \mbox{ otherwise }
\end{array} \right.
\end{equation}
Let $E_{i,j}$ denote the $(n+1) \times (n+1)$ matrix with a $1$ in entry $(i,j)$ and zeros everywhere else. Then the commutation relations \ref{commNn} imply that the map
\begin{align*}
\phi: \n_{\F}  & \rightarrow \on{Mat}_{n+1} \\
\phi(  \delta_{L(k, \cdots, k+m)} ) & = - E_{k, k+m+1}
\end{align*}
is an isomorphism of $\n_{\F}$ onto the Lie algebra of upper-triangular $(n+1) \times (n+1)$ matrices. 
\end{example}

\begin{example}
Let $S = \{ 1, 2\}$, and let $\F=\overline{\F'}$, where $\F'$ consists of all $S$--colored ladders where the colors alternate. 
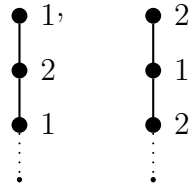
\psset{levelsep=4ex, treesep=0.6cm, treenodesize=1pt,nodesepB=-2pt}
\[
\pstree{\Tc*{3pt}~[tnpos=r]{1}}{ \pstree{\Tc*{3pt}~[tnpos=r]{2}}{ \pstree{\Tc*{3pt}~[tnpos=r]{1}}{ \psset{linestyle=dotted} \Tc*{1pt}~[tnpos=r]{}}}} ,\hspace{1cm} \pstree{\Tc*{3pt}~[tnpos=r]{2}}{ \pstree{\Tc*{3pt}~[tnpos=r]{1}}{ \pstree{\Tc*{3pt}~[tnpos=r]{2}}{ \psset{linestyle=dotted} \Tc*{1pt}~[tnpos=r]{}}}}
\]
Let us denote by $L(i,n), \; i \in S, n \geq 1$ the alternating ladder with $n$ vertices, whose root is colored $i$. Then $\n_{\F} = \on{span}\{ L(i,n) \}, \; i \in S, n \geq 1$. We have
\begin{align*}
\delta_{L(i,n)} \rhd \delta_{L(i,m)} & =  \left\{ \begin{array}{ll} 
  \delta_{L(i,n+m)} & \mbox{ if $m \equiv 0 \; mod \;  2, \; i \in S $} \\
  0 & \mbox{ otherwise }
\end{array} \right. \\
\delta_{L(i,n)} \rhd \delta_{L(j,m)} & =  \left\{ \begin{array}{ll} 
  \delta_{L(j,n+m)} & \mbox{ if $m \equiv 1 \; mod \;  2, \; i \neq j \in S$ } \\
  0 & \mbox{ otherwise }
\end{array} \right.
\end{align*}

It follows that 
\begin{align} \label{cr_gl2hat}
[\delta_{L(i,2k)}, \delta_{L(j,2l)}] & = 0 \\
\nonumber [\delta_{L(i,2k)}, \delta_{L(j,2l+1)}] & =  \left\{ \begin{array}{ll}  - \delta_{L(j, 2(k+l) +1 )} & \mbox{if $ i = j$} \\
 \delta_{L(j, 2(k+l) +1 )} & \mbox{if $ i \neq j$}  \end{array} \right. \\
\nonumber [\delta_{L(i,2k+1)}, \delta_{L(j,2l+1)}] &= \delta_{L(j,2(k+l+1))} -  \delta_{L(i,2(k+l+1))} 
\end{align} 

Recall that $\mathfrak{gl}_2 = \on{Mat}_2 = \n_{-} \oplus \h \oplus \n_+$, where $$\n_- = \on{span}\{f \}, \n_+ = \on{span}\{ e \}, \h = \on{span}\{ h_1, h_2 \}$$ and
$$ f = \left( \begin{matrix} 0 & 0 \\ 1 & 0 \end{matrix} \right) \hspace{0.5cm} h_1 =  \left( \begin{matrix} 1 & 0 \\ 0 & 0 \end{matrix} \right) \hspace{0.5cm} h_2 =  \left( \begin{matrix} 0 & 0 \\ 0 & 1 \end{matrix} \right)   \hspace{0.5cm} e = \left( \begin{matrix} 0 & 1 \\ 0 & 0 \end{matrix} \right) $$
Let $L\mathfrak{gl}_2 = \mathfrak{gl}_2 \otimes \mathbb{Q}[t,t^{-1}]$ be the loop algebra of $\mathfrak{gl}_2 $, with bracket
\[
[X\otimes t^m, Y \otimes t^n] = [X,Y] \otimes t^{n+m}
\]
$L\mathfrak{gl}_2 $ also has a triangular decomposition $L\mathfrak{gl}_2 = L\mathfrak{gl}^+_2 \oplus \h \oplus L\mathfrak{gl}^-_2$, where
\[
L\mathfrak{gl}^+_2 = \n_+ \oplus \mathfrak{gl}_2\otimes t \mathbb{Q}[t] \hspace{1cm} L\mathfrak{gl}^-_2 = \n_- \oplus \mathfrak{gl}_2 \otimes t^{-1} \mathbb{Q}[t^{-1}] 
\]
Let 
\begin{align*}
\phi: \n_{\F} & \rightarrow L\mathfrak{gl}^+_2 \\
\phi( \delta_{L(1,2k+1)}) & = e \otimes t^k \\
\phi( \delta_{L(2, 2k+1)}) &= f \otimes t^{k+1} \\
\phi( \delta_{L(1,2k) }) &= - h_1 \otimes t^k \\
\phi( \delta_{L(2,2k) }) &= - h_2 \otimes t^{k}
\end{align*}
\end{example}
It follows from \ref{cr_gl2hat} that $\phi$ is an isomorphism. It follows that $U(L\mathfrak{gl}^+_2 )$ has an integral basis which may be identified with Young diagrams whose columns are colored by alternating strings of $1$'s and $2$'s. 

\begin{example}
A straightforward generalization of the previous example, with $S=\{ 1, \cdots, n\}$ and $\F'$ consisting of ladders periodically colored by $1, \cdots, n$ yields $\n_\F \simeq L \mathfrak{gl}^+_n$. 
\end{example}

\begin{example}
Let $S = \{1,2 \}$, and let $\F=\overline{\F'}$, where $\F'$ is the set of all ladders colored by a sequence of $1$'s followed by a sequence of $2$'s. 
\psset{levelsep=4ex, treesep=0.6cm, treenodesize=1pt,nodesepB=-2pt}
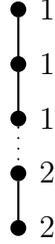
\[
\pstree{\Tc*{3pt}~[tnpos=r]{1}}{ \pstree{\Tc*{3pt}~[tnpos=r]{1}}{ \pstree{\Tc*{3pt}~[tnpos=r]{1}}{ \psset{linestyle=dotted}   \pstree{\Tc*{3pt}~[tnpos=r]{2}}{ \psset{linestyle=solid}  \pstree{\Tc*{3pt}~[tnpos=r]{2}}{}    }      }} }
\]
Denote by $L(i,j)$ the ladder with $i$ $1$'s followed by $j$ $2$'s. We have
\begin{align*}
\delta_{L(i,j)} \rhd \delta_{L(m,n)} & = 0 \mbox{ if $ij > 0$ and $mn > 0$} \\
\delta_{L(i,0)} \rhd \delta_{L(m,n)} & =  \left\{ \begin{array}{ll} 
  \delta_{L(i+m,0)} & \mbox{ if $n=0$} \\
  0 & \mbox{ otherwise }
\end{array} \right. \\
\delta_{L(0,j)} \rhd \delta_{L(m,n)} &= \delta_{L(m,n+j)} \\
\delta_{L(i,j)} \rhd \delta_{L(m,0)} &= \delta_{L(i+m,j)} \\
\delta_{L(i,j)} \rhd \delta_{L(0,n)} & =  \left\{ \begin{array}{ll} 
  \delta_{L(0,j+n)} & \mbox{ if $i=0$} \\
  0 & \mbox{ otherwise }
\end{array} \right. \\
\end{align*}
so that we obtain the following non-zero commutation relations (i.e. all other commutators are $0$):
\begin{align*}
[\delta_{L(i,0)}, \delta_{L(0,n)}] &= - \delta_{L(i,n)} \\
[\delta_{L(i,0)}, \delta_{L(m,n)}] &= - \delta_{L(m+i,n)} \mbox{  if $n > 0$} \\
[\delta_{L(0,j)}, \delta_{L(m,n)}] &= \delta_{L(m,n+j)} \mbox{  if $m > 0$}
\end{align*}
\end{example}

\begin{example}
Let $S=\{1, 2, \cdots, n \}$, and let $\F = \overline{\F'}$, where $\F'$ consists of all $S$--colored corollas (rooted trees where all leaves are connected directly to the root)
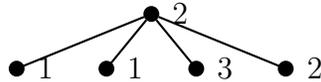
\psset{levelsep=4ex, treesep=0.6cm, treenodesize=1pt,nodesepB=-2pt}
\[
\pstree{\Tc*{3pt}~[tnpos=r]{2}}{ \pstree{\Tc*{3pt}~[tnpos=r]{1}}{} \pstree{\Tc*{3pt}~[tnpos=r]{1}}{} \pstree{\Tc*{3pt}~[tnpos=r]{3}}{} \pstree{\Tc*{3pt}~[tnpos=r]{2}}{}  }
\]
Closing $\F'$ with respect to convex sub-posets means adjoining singleton colored trees. Denote by $X(i)$ the singleton tree colored by $1 \leq i \leq n$, and by $Y(i, a_1, \cdots, a_n)$ the corolla whose root is colored $i$ and which has $a_1 + a_2 \cdots + a_n$ leaves, with $a_1$ colored $1$, $a_2$ colored $2$ etc. In $\n_{\F}$ we have
\begin{align*}
\delta_{X(i)} \rhd  \delta_{X(j)} & = \delta_{Y(j, 0, \cdots, \underset{i}{1}, \cdots, 0)} \\
\delta_{X(i)} \rhd \delta_{Y(j,a_1, \cdots, a_n)} &= \delta_{Y(j, a_1, \cdots, a_i + 1, \cdots, a_n)} \\
\delta_{Y(j,a_1, \cdots, a_n)} \rhd \delta_{X(i)} & = 0 \\
\delta_{Y(j,a_1, \cdots, a_n)} \rhd \delta_{Y(j,b_1, \cdots, b_n)} &= 0
\end{align*}
which leads to the following commutation relations:
\begin{align*}
[\delta_{X(i)}, \delta_{X(j)}] &=  \delta_{Y(j, 0, \cdots, \underset{i}{1}, \cdots, 0)}  - \delta_{Y(i, 0, \cdots, \underset{j}{1}, \cdots, 0)} \\
[\delta_{X(i)}, \delta_{Y(j,a_1, \cdots, a_n)} ] &= \delta_{Y(j, a_1, \cdots, a_i + 1, \cdots, a_n)} \\
[\delta_{Y(j,a_1, \cdots, a_n)}, \delta_{Y(j,b_1, \cdots, b_n)}] &= 0 
\end{align*}
\end{example}

\end{document}